\documentclass[11pt]{article}
\usepackage[T1]{fontenc}
\usepackage{lmodern}
\usepackage{microtype}
\usepackage{amsmath,amssymb,amsthm,mathtools}
\usepackage{booktabs}
\usepackage{array}
\usepackage{enumitem}
\usepackage[margin=1in]{geometry}
\usepackage[hidelinks]{hyperref}
\usepackage{xcolor}
\usepackage{stmaryrd}
\usepackage{mathrsfs}
\usepackage[nameinlink,capitalise,noabbrev]{cleveref}

\usepackage{booktabs}
\usepackage{tabularx}
\usepackage{array}
\usepackage{placeins}
\usepackage{tikz}

\hypersetup{
  colorlinks=true,
  linkcolor=blue!50!black,
  citecolor=blue!50!black,
  urlcolor=blue!50!black
}

\newtheorem{theorem}{Theorem}[section]
\newtheorem{proposition}[theorem]{Proposition}
\newtheorem{conjecture}[theorem]{Conjecture}
\newtheorem{problem}[theorem]{Problem}
\newtheorem{lemma}[theorem]{Lemma}
\newtheorem{corollary}[theorem]{Corollary}
\theoremstyle{definition}
\newtheorem{definition}[theorem]{Definition}
\newtheorem{remark}[theorem]{Remark}
\newtheorem{maintheorem}{Theorem}

\newtheorem{maincorollary}[maintheorem]{Corollary}
\newcolumntype{Y}{>{\raggedright\arraybackslash}X}
\newcommand{\R}{\mathbb{R}}
\newcommand{\Sph}{\mathbb{S}}
\newcommand{\vol}{\operatorname{vol}}
\newcommand{\conv}{\operatorname{conv}}

\newcommand{\depth}{\operatorname{depth}}
\newcommand{\cK}{c}
\newcommand{\dd}{\,\mathrm{d}}
\newcommand{\ip}[2]{\left\langle #1,#2\right\rangle}
\newcommand{\eps}{\varepsilon}

\title{\textbf{Barycentric Cuts at Maximal Depth and a Five-Cut Theorem}}
\author{Xiaoxiang Jiao \and Hangyue Zhu}
\date{\today}

\begin{document}
\maketitle

\begin{abstract}
We determine the sharp number and the complete finite spectrum of barycentric hyperplanes through the Tukey median. We further prove that every convex body in $\mathbb R^n$, $n\ge3$, contains an interior point incident with at least five barycentric hyperplanes. In particular, this settles Grünbaum's original conjecture in dimension four.
\end{abstract}



\tableofcontents
\section{Introduction}
\label{sec:introduction}

Throughout, let $n\geq2$. A \emph{convex body} is a compact convex set $K\subset\mathbb R^{n}$ with nonempty interior. Let $\mathcal K^n$ denote the class of convex bodies in $\mathbb R^n$. The \emph{volume} and \emph{centroid} of a convex body $K$ are
\[
\vol(K)=\int_K 1\,dx,
\qquad
c(K)=\frac{1}{\vol(K)}\int_K x\,dx.
\]
If $H$ is an affine hyperplane with $\vol_{n-1}(K\cap H)>0$, the centroid of the section $K\cap H$ is
\[
c(K\cap H)
=
\frac{1}{\vol_{n-1}(K\cap H)}
\int_{K\cap H}x\,d\vol_{n-1}(x).
\]
For $p\in\operatorname{int}K$, a hyperplane $H$ passing through $p$ is said to be \emph{barycentric at $p$} if \(c(K\cap H)=p.\) We write
\[
\mathcal B_K(p)
:=
\{H:p\in H,\ c(K\cap H)=p\},
\]
where hyperplanes are counted without orientation. We call $K$ \emph{smooth and positively curved} if $\partial K$ is $C^\infty$ and all its principal curvatures are positive.

In 1961, Gr\"unbaum posed two closely related questions concerning barycentric sections of convex bodies \cite{Gru61}; see also \cite{PTW}. We shall refer to the affirmative form of his general-point question as the following conjecture.

\begin{conjecture}[Gr\"unbaum's conjecture]
\label{conj:grunbaum-general}
For every convex body $K\subset\mathbb R^n$, there exists $p\in\operatorname{int}K$ such that 
\(
\#\mathcal B_K(p)\geq n+1.
\)
\end{conjecture}

Two canonical choices of the point $p$ are the centroid of $K$ and the point of maximal halfspace depth. These choices lead to two quantitative versions of Conjecture~\ref{conj:grunbaum-general}.

For the centroid version, define
\[
\mu(K):=\#\mathcal B_K(c(K)),
\qquad
\mu(n):=\min_{{K\in\mathcal K^n}}\mu(K),
\]
where the minimum is taken over all convex bodies in $\mathbb R^n$.

\begin{problem}[{The centroid problem of Gr\"unbaum and Loewner; \cite[Problem~28]{Fen67}}]
\label{prob:centroid}
Determine $\mu(n)$. In particular, is $\mu(n)\geq n+1$?
\end{problem}

It is now known that
\[
\mu(2)=3,
\qquad
\mu(n)=1\quad\text{for every }n\geq3.
\]
Thus the proposed lower bound $n+1$ is false for the centroid in every dimension $n\geq3$ \cite{XiongYang,MTY25,Gru61,Fen67}.

We next turn to the second distinguished point considered in this paper. For $p\in\operatorname{int}K$, its \emph{halfspace depth} with respect to the uniform probability measure on $K$ is
\[
\operatorname{depth}_K(p)
:=
\frac{1}{\operatorname{vol}(K)}
\min_{u\in\mathbb S^{n-1}}
\operatorname{vol}
\bigl(K\cap\{x:\langle x-p,u\rangle\geq 0\}\bigr).
\]
This function has a unique maximizer $p_K\in\operatorname{int}K$, called the \emph{Tukey median}, or \emph{halfspace median}, of $K$ \cite[Introduction]{PTW}, see also \cite[Section~3.3]{NagySurvey}. Halfspace depth was introduced by Tukey as an affine-invariant multivariate analogue of univariate ranks and medians \cite{Tukey}, its statistical robustness and geometric properties were subsequently developed in \cite{DG92,RR99}. A recent account can be found in \cite{NagySurvey}.

If $u$ realizes the minimum in $\operatorname{depth}_K(p_K)$, Dupin's theorem implies that the hyperplane $p_K+u^\perp$ is barycentric at $p_K$ \cite[Section~3.3]{NagySurvey}. This connection motivates the following maximal-depth version of the barycentric-hyperplane problem. Set
\[
\nu(K):=\#\mathcal B_K(p_K),
\qquad
\nu(n):=\min_{K\in\mathcal K^n}\nu(K),
\]
and define the finite spectrum by
\[
\Sigma_n^{\mathrm{fin}}
:=
\left\{
m\in\mathbb N:
\nu(K)=m
\ \text{for some convex body }K\subset\mathbb R^n
\right\}.
\]

Grünbaum claimed that $\nu(n)\geq n+1$ \cite[\S6.2]{Gru63}. Patáková, Tancer, and Wagner later identified a gap in his argument and proved the universal bound
\[
\nu(n)\geq 4,\qquad n\geq 3
\]
\cite{PTW}. This leads to the following problem.

\begin{problem}[The maximal-depth problem]\label{prob:maximal-depth}
Determine $\nu(n)$. More generally, determine
$\Sigma_n^{\mathrm{fin}}$.
\end{problem}

The principal developments relevant to the two problems above are summarized in Table~\ref{tab:related-work}.

\begin{table}[htbp]
\centering
\caption{Selected developments concerning barycentric hyperplanes.}
\label{tab:related-work}
\small
\setlength{\tabcolsep}{6pt}
\renewcommand{\arraystretch}{1.15}
\begin{tabularx}{\textwidth}{
    @{}
    p{0.11\textwidth}
    >{\raggedright\arraybackslash}X
    p{0.22\textwidth}
    @{}
}
\toprule
Year & Statement/result & Reference \\
\midrule

1961
&
The general and centroid barycentric-section problems were posed.
&
\cite{Gru61}
\\

1963
&
It was claimed that \(\nu(n)\ge n+1\), the argument was later shown to contain a gap.
&
\cite{Gru63,PTW}
\\

1967
&
\(\mu(2)=3\).
&
\cite{Fen67}
\\

2022
&
\(\nu(n)\ge 4\) for every \(n\ge 3\); $\nu(2)\geq 3$. This implies that Conjecture \ref{conj:grunbaum-general} is true for dimensions $2, 3$.
&
\cite{PTW}
\\

2025
&
\(\mu(n)=1\) for every \(n\ge 5\).
&
\cite{MTY25}
\\

2026
&
\(\mu(n)=1\) for every \(n\ge 3\).
&
\cite{XiongYang}
\\

{\ }
&
\(\nu(n)=4\) for every \(n\ge 3\); $\nu(2)=3.$
&
Theorem \ref{thm:intro-spectrum}
\\
{\ }
&
For dimensions \(n\ge3\): at least five barycentric cuts at some interior point.
&
Theorem \ref{intro-5}

\\

{\ }
&
Conjecture \ref{conj:grunbaum-general} is true for dimension $4$.
&
Corollary \ref{intro-conj}
\\

\bottomrule
\end{tabularx}
\end{table}
\FloatBarrier

\medskip
\noindent
\textbf{Main results.}
Our first result gives a sharp answer to Problem~\ref{prob:maximal-depth} and, more generally, classifies all possible finite cut numbers at the Tukey median.

\begin{maintheorem}[Theorems \ref{thm:nu-four}, \ref{thm:spectrum}; Remarks \ref{rem:planar-cut}, \ref{rem:planar-spectrum}]
\label{thm:intro-spectrum}
For \(n=2\),
\[
    \nu(2)=3,
    \qquad
    \Sigma^{\mathrm{fin}}_2=\{3,4,5,\ldots\};
\]
For every \(n\ge3\),
\[
    \nu(n)=4,
    \qquad
    \Sigma^{\mathrm{fin}}_n=\{4,5,6,\ldots\}.
\]
Moreover, every finite value in this spectrum is realized by a $C^\infty$ smooth, positively curved convex body that may be chosen arbitrarily close to unit ball. 
\end{maintheorem}

In particular, when $n\geq4$, the point of maximal depth need not be incident with $n+1$ barycentric hyperplanes. Thus the maximal-depth strengthening of Conjecture~\ref{conj:grunbaum-general} is false.

We now combine the lower bound of \cite{PTW} with a local degree argument to obtain a uniform free-point result.
\begin{maintheorem}[Theorem \ref{thm:five-cut}]\label{intro-5}
Every convex body $K\subset\R^n$, $n\ge3$, admits a point $p\in\operatorname{int}K$ through which pass at least five barycentric hyperplanes.
\end{maintheorem}

As an immediate corollary, combining Theorem \ref{intro-5} and \cite{PTW},we also have the following positive result for the original free-point conjecture in low dimensions.

\begin{maincorollary}[Corollary \ref{cor:grunbaum-four}]\label{intro-conj}

Conjecture \ref{conj:grunbaum-general} is true in dimensions $2\le n\le 4$. 
\end{maincorollary}

Finally, we prove a centered exact-realization theorem in Section \ref{sec:centered}. In contrast with the perturbative realization used in \cite{XiongYang}, the prescribed halfspace-volume functional and the centroid constraint are both satisfied exactly. We also show that the absence of the degree-one spherical-harmonic component is necessary for differentiable near-ball realizations.

\medskip
\noindent
\textbf{Organization of the paper.}
Section~\ref{sec:profiles} recalls the hemispherical transform and the critical-point characterization of barycentric sections, and proves the exact realization theorem. Section~\ref{sec:maximal-depth} proves the sharp four-cut theorem and determines the finite cut spectrum. Section~\ref{sec:five-cut} derives the existence of a point on at least five barycentric hyperplanes in every dimension \(n\ge3\). Section~\ref{sec:centered} establishes the centered exact realization theorem and its converse.

\section{Halfspace volume functional and exact realization}\label{sec:profiles}
We now translate the geometric section problem into an analytic problem on the sphere. The relevant object is the halfspace volume functional and its representation by the hemispherical transform.
\subsection{The hemispherical transform and barycentric sections}
We first recall the distinguished point at which the cut number will be measured. 

For $p\in\R^n$ and $\xi\in\Sph^{n-1}$, set
\[
 H^+_{p,\xi}
   :=\{x\in\R^n:\ip{x-p}{\xi}\geq 0\},
 \qquad
 V_{K,p}(\xi):=\vol(K\cap H^+_{p,\xi}).
\]
Thus
\[
    \operatorname{depth}_K(p)
    =
    \frac{1}{\operatorname{vol}(K)}
    \min_{\xi\in\mathbb S^{n-1}}V_{K,p}(\xi).
\]
We retain the notation \(p_K\) for the unique Tukey median introduced in the Introduction. Translations reduce questions at $p$ to questions at the origin. When $0\in\operatorname{int}K$, we abbreviate
\(
             D_K(\xi):=V_{K,0}(\xi).
\)

For the rest of this paper, let $\dd\sigma$ denote spherical Lebesgue measure on $\Sph^{n-1}$, and let $\omega_n=\vol(B^n)$, so that $\sigma(\Sph^{n-1})=n\omega_n$.

\begin{definition}
A \emph{spherical harmonic of degree $k$} is the restriction to $\Sph^{n-1}$ of a homogeneous harmonic polynomial of degree $k$ on $\R^n$. The degree-one spherical harmonics are precisely $u\mapsto\ip{a}{u}$, $a\in\R^n$. The degree-one component of $f\in L^2(\Sph^{n-1})$ vanishes if its orthogonal projection onto this space is zero, equivalently,
\[
                \int_{\Sph^{n-1}}u f(u)\dd\sigma(u)=0.
\]
\end{definition}

\begin{definition}
For $f\in L^1(\Sph^{n-1})$, its \emph{hemispherical transform} is
\[
 (Tf)(\xi)
   :=\int_{\{u\in\Sph^{n-1}:\ip{u}{\xi}\geq 0\}}
          f(u)\dd\sigma(u).
\]
\end{definition}

The following multiplier and range properties are the only facts about the hemispherical transform needed below.

\begin{theorem}[{\cite[(3.19) and Theorem~5.1]{Rubin14}, see also \cite{Rubin1999}}]\label{thm:rubin}
For \(n\ge2\), the hemispherical transform is invertible on
$C^\infty_{\mathrm{odd}}(\Sph^{n-1})$.  It preserves every odd spherical harmonic subspace and has a nonzero eigenvalue on each of them. In particular, it preserves the condition that the degree-one component vanish.
\end{theorem}

If $h$ is even, antipodal pairing gives
\begin{equation}\label{eq:T-even}
             Th=\frac12\int_{\Sph^{n-1}}h\dd\sigma.
\end{equation}
Thus every smooth even mean-zero function lies in $\ker T$.

The following proposition provides the basic bridge between halfspace-volume and barycentric sections.

\begin{proposition}[{\cite[Proposition~1.11]{PTW}; \cite[Lemma~2.1]{XiongYang}}]
\label{prop:known-bridge}
Let $K$ be a convex body with $0\in\operatorname{int}K$. Then \(D_K\in C^1(\mathbb S^{n-1})\), and
\begin{equation}\label{eq:profile-transform}
 D_K(\xi)
   =\frac1n\int_{\ip{u}{\xi}\geq 0}\rho_K(u)^n\dd\sigma(u)
   =\frac1nT(\rho_K^n)(\xi).
\end{equation}
Moreover, $\xi$ is a critical point of $D_K$ if and only if
$\xi^\perp$ is barycentric at the origin.
\end{proposition}
Notice that $\xi$ and $-\xi$ determine the same unoriented hyperplane, and that
\[
        D_K(-\xi)=\vol(K)-D_K(\xi).
\]

\subsection{Exact realization near the Euclidean ball}
Because the functional is linear in $q=\rho_K^n$, the preceding identity allows one to prescribe it exactly near the Euclidean ball. Here and below, $K_\varepsilon\to B^n$ in $C^\infty$ means that $\rho_{K_\varepsilon}\to1$ in $C^\infty(\mathbb S^{n-1})$. 
%加句话说明这里和熊-杨的差一个高阶小
\begin{proposition}
\label{thm:exact-local}
Let $n\geq 2$ and $F\in C^\infty_{\mathrm{odd}}(\Sph^{n-1})$. There is $\eps_0>0$ such that, for every $0<|\eps|<\eps_0$, there is a smooth, positively curved convex body $K_\eps$ satisfying
\[
 \vol(K_\eps)=\omega_n,\qquad
 K_\eps\longrightarrow B^n\quad\text{in }C^\infty,
\]
and for every $\xi\in\Sph^{n-1}$,
\begin{equation}\label{eq:exact-profile}
             D_{K_\eps}(\xi)
               =\frac{\omega_n}{2}+\eps F(\xi).
\end{equation}
More explicitly, if $g=T^{-1}F$, one may take
\begin{equation}\label{eq:exact-radial}
       \rho_{K_\eps}(u)
          =\bigl(1+n\eps g(u)\bigr)^{1/n}.
\end{equation}
\end{proposition}

\begin{proof}
By Theorem~\ref{thm:rubin}, $g=T^{-1}F$ is smooth and odd. For sufficiently small $|\eps|$, the function $q_\eps:=1+n\eps g$ is positive. Define $\rho_\eps=q_\eps^{1/n}$ and let $K_\eps$ be its radial body. Since $\rho_\eps\to 1$ in $C^\infty$, the radial hypersurface $\rho_\eps(u)u$ is a $C^\infty$-small perturbation of the unit sphere. 

The second fundamental form depends continuously on the embedding and its first two derivatives, hence the second fundamental form $A^{K_\eps}\to A^{\mathbb S^{n-1}}$ in $C^0$. The unit sphere has positive principal curvatures \(\kappa_i=1\), and positive definiteness is preserved under sufficiently small perturbations. Therefore, for sufficiently small \(|\varepsilon|\), all principal curvatures of \(\partial K_\varepsilon\) remain positive, so \(K_\varepsilon\) is a smooth positively curved convex body. For similar discussion, see \cite[Lemma~2.4]{XiongYang} or \cite{Schneider}.

Oddness of $g$ gives
\[
 \vol(K_\eps)
  =\frac1n\int_{\Sph^{n-1}}q_\eps\dd\sigma=\frac{1}{n}\sigma({\mathbb S^{n-1}})+\eps \int_{\mathbb S^{n-1}}g \dd\sigma=\omega_n.
\]
Finally, \eqref{eq:profile-transform} and $Tg=F$ give the identity
\[
 D_{K_\eps}
  =\frac1nT(1+n\eps g)
  =\frac{\omega_n}{2}+\eps F.
\]\end{proof}

\section{Barycentric cuts at the Tukey median}\label{sec:maximal-depth}

\subsection{The sharp four-cut construction}
We now impose the additional requirement that the prescribed origin be the Tukey median of the realizing body.
%说一下这个是minimum criterion

\begin{lemma}\label{lem:balance}
Let $0\in \operatorname{int} K$, and suppose $D_K(\xi)=C+\eps F(\xi)$, where $\eps>0$, and $F$ continuous. Let $U=\operatorname{argmin}_{\mathbb S^{n-1}}F=\{u:F(u)=\min_{\Sph^{n-1}}F\}$. If $0\in\conv U$, then it is a maximal depth point of $K$.
\end{lemma}

\begin{proof}
Fix $p\in\operatorname{int}K$. Since $0\in\conv U$, there exist finitely many $u_1,\cdots, u_m\in U$ and $\lambda_i>0$ with $\sum_i \lambda_i=1$, such that $\sum_{i=1}^m \lambda_iu_i=0$. Taking the inner product with $p$ gives $\sum_{i=1}^m \lambda_i\langle p,u_i\rangle=0$. Hence at least one \(u_i\) satisfies \(\langle p,u_i\rangle\ge0\). Then
\begin{equation}\label{eq: inclusion of half spaces}
   H^+_{p,u}=\{x:\ip{x}{u}\geq\ip{p}{u}\}
       \subseteq \{x:\ip{x}{u}\geq 0\}=H^+_{0,u}.
\end{equation}
Consequently,
\[
 \vol(K)\depth_K(p)
   \leq \vol(K\cap H^+_{p,u})
   \leq D_K(u)
   =\min_\xi D_K(\xi)
   =\vol(K)\depth_K(0).
\]The first inequality follows by definition, the second comes from \eqref{eq: inclusion of half spaces}, and the third equality is by our choice of $u$. Thus $0$ maximizes the halfspace depth. By the uniqueness of the Tukey median, $p_K=0$.
\end{proof}
It therefore remains to construct a smooth odd function with exactly four antipodal pairs of critical points and a balanced set of global minimizers. To this end, we use an explicit polynomial variant of the depth-like construction introduced in \cite[Section~5]{PTW}. The additional ingredient in our argument is to realize this function exactly as a perturbation of the halfspace-volume functional of a smooth, positively curved convex body, while ensuring that the origin remains its Tukey median.

\begin{proposition}
\label{prop:smooth-four}
Let $n\geq 3$, write
\begin{equation}\label{eq: decomposition}
       \R^n=\R^{\,n-2}_y\oplus\R^2_z,
       \qquad r=\lVert y\rVert,
\end{equation}
and identify $z=(z_1,z_2)$ with $z_1+iz_2$.  Choose $a>0$ and
$b>5a/2$.  The restriction to $\Sph^{n-1}$ of
\begin{equation}\label{eq:F4}
 F_4(y,z)
   :=a(5r^2-3r^4)y_1+b\,\operatorname{Re}z^3
\end{equation}
is smooth, odd and has exactly four antipodal critical pairs:
\[
 (\pm e_1,0)
 \quad\text{and}\quad
 (0,e^{ik\pi/3}),\qquad k=0,\ldots,5.
\]
Its three global minima are
\begin{equation}\label{eq:balanced-minima}
 U_4=
 \left\{(0,e^{i(2j+1)\pi/3}):j=0,1,2\right\}.
\end{equation}
\end{proposition}

\begin{proof}
The formula is the restriction of a smooth polynomial and is odd by direct inspection. Put $s=\lVert z\rVert$, so $r^2+s^2=1$. When $r,s>0$, write $y=r\eta$, $\eta\in\Sph^{n-3}$, and $z=se^{i\theta}$. Then
\[
 F_4=a(5r^3-3r^5)\eta_1+bs^3\cos(3\theta).
\]

Since \(5r^3-3r^5=r^3(5-3r^2)>0\), the critical point condition in the \(\eta\)-direction is the critical point condition for the height function
\(
\eta\mapsto \eta_1
\) on the sphere \(S^{n-3}\). Only possibilities are $\eta=\sigma e_1$, $\sigma\in\{-1,1\}$. Similarly, differentiating \(\theta\), we obtain
\(
\partial_\theta F_4=-3bs^3\sin(3\theta)
\). Therefore 
\(
\cos(3\theta)=c\in\{-1,1\}
\). For \(n=3\), \(S^{n-3}=S^0=\{\pm1\}\), thus the same conclusion holds trivially.

Differentiate in $r$, with $s=(1-r^2)^{1/2}$.  Since
\[
 \frac{\dd}{\dd r}(5r^3-3r^5)=15r^2s^2,
 \qquad
 \frac{\dd}{\dd r}s^3=-3rs,
\]
the radial critical equation is
\[
       15a\sigma r^2s^2-3bcrs=0,
       \qquad\text{hence}\qquad
       bc=5a\sigma rs.
\]
But $rs\leq 1/2$ while $b>5a/2$, so no mixed critical point of $r,s$ both positive exists.

It remains to consider the cases $r=0$ and $s=0$. Suppose first that $r=0$. Then $y=0$, $|z|=1$, \(F_4=b\cos(3\theta)\), and
\[
T_{(0,z)}S^{n-1}
=
\mathbb R^{n-2}\oplus z^\perp .
\]
Since
\[
\dd_y\!\left[(5|y|^2-3|y|^4)y_1\right]_{y=0}=0,
\]
the differential of $F_4$ vanishes in all transverse $y$-directions. Hence $(0,z)$ is critical on $S^{n-1}$ if and only if $z$ is a critical point of $z\mapsto b\operatorname{Re}(z^3)$ on $S^1$. The critical points are $e^{ik\pi/3}$, $k=0,\ldots,5$.

Suppose next that $s=0$. Then $z=0$, $|y|=1$, \(F_4=2ay_1\), and
\[
T_{(y,0)}S^{n-1}
=
y^\perp\oplus\mathbb R^2 .
\]
Since
\[
\dd_z\operatorname{Re}z^3\big|_{z=0}=0,
\]
the differential vanishes in all transverse $z$-directions. The height function $y\mapsto y_1$ has precisely the two critical points $y=\pm e_1$. Hence the only critical points with $s=0$ are $(\pm e_1,0)$. This also covers the case $n=3$, where the corresponding coordinate sphere is $\mathbb S^0=\{\pm e_1\}$.

We have therefore found all critical points of $F_4$:\[
\operatorname{Crit}(F_4)=\{(\pm e_1,0)\}\cup\{(0,e^{ik\pi/3}):k=0,\ldots,5\}.
\] 
Their corresponding critical values are $\pm b$ at the six $z$-points and $\pm2a$ at the two $y$-points. Since $S^{n-1}$ is compact, every global extremum is a critical point. Because $b>5a/2>2a$, the global minimum equals $-b$ and is attained precisely at
\[
U_4=\{(0,e^{i(2j+1)\pi/3}):j=0,1,2\}.
\]
The three corresponding $z$-coordinates are the roots of $z^3=-1$ and have sum zero.
\end{proof}

Combining the universal lower bound from \cite{PTW} with the preceding model gives the sharp result.

\begin{theorem}\label{thm:nu-four}
For every $n\geq 3$, \(\nu(n)=4.\) Moreover, the extremal body may be chosen \(C^\infty\)-smooth, strictly convex, arbitrarily \(C^\infty\)-close to the unit ball, and with exactly three depth-realizing hyperplanes through its Tukey median.
\end{theorem}

\begin{proof}
The lower bound $\nu(K)\geq 4$ for every convex body in dimension $n\geq 3$ comes from \cite[Theorem 1.10]{PTW}.

For the upper bound, apply Proposition~\ref{thm:exact-local} to the smooth odd function $F_4$ from Proposition~\ref{prop:smooth-four}, with $\eps>0$. This gives a smooth positively curved body $K_\eps$ with
\[
       D_{K_\eps}=\frac{\omega_n}{2}+\eps F_4.
\]
The three minimizing directions in \eqref{eq:balanced-minima} have the origin in their convex hull. Lemma~\ref{lem:balance} therefore shows that $0=p_{K_\eps}$. By Proposition~\ref{prop:known-bridge}, the barycentric hyperplanes through the origin correspond exactly to antipodal pairs of critical points of $D_{K_\eps}$, equivalently of $F_4$. There are exactly four such pairs.

Moreover, \(\operatorname{argmin}D_{K_\varepsilon}=\operatorname{argmin}F_4=U_4\). Its three members determine three distinct unoriented hyperplanes which realize the depth at the origin.
\end{proof}
\begin{remark}[The planar case]\label{rem:planar-cut}
When \(n=2\), write \(z=e^{i\theta}\in\mathbb S^1\) and set
\[
    F_3(e^{i\theta})=\cos(3\theta).
\]
This is odd and has exactly three antipodal critical pairs. Its minimum set is
\[
    U_3=
    \left\{
        e^{i(2j+1)\pi/3}:j=0,1,2
    \right\},
\]
whose three elements have sum zero. Proposition~\ref{thm:exact-local} and Lemma~\ref{lem:balance} therefore give a convex body with exactly three barycentric lines through its Tukey median. Together with \cite[Corollary~1.9]{PTW}, which gives $\nu(2)\geq 3$, this yields
\[
    \nu(2)=3.
\]
This is also the original motivation of \eqref{eq: decomposition}.
\end{remark}

\subsection{The finite cut spectrum}
The preceding construction realizes the smallest possible cut number. To obtain every larger finite value, we insert antipodal birth--death critical points away from the minimizing set.
\begin{definition}
Let $f$ be a smooth function on a smooth manifold $M$. A point $x$ is \emph{critical} if $\dd f_x=0$.  It is a \emph{birth--death critical point} if, in suitable smooth local coordinates centered at $x$, the function has the form
\[
 f(x)+t_1^3+\epsilon_2t_2^2+\cdots+\epsilon_dt_d^2,
 \qquad \epsilon_j\in\{-1,1\}.
\]
Such a point is isolated and degenerate. It is the single critical point present at the transition parameter at which a cancelling pair of Morse critical points is born or dies.
\end{definition}

\begin{lemma}\label{lem:insertion}
Let $M$ be a smooth manifold, let $f\in C^\infty(M)$, and let $x\in M$ be a regular point. Given a neighborhood $U$ of $x$ and $\delta>0$, there is $\widetilde f\in C^\infty(M)$ such that
\begin{enumerate}
\item $\widetilde f=f$ outside $U$;
\item $\lVert\widetilde f-f\rVert_{C^0}<\delta$;
\item $\operatorname{Crit}(\widetilde f)=\operatorname{Crit}(f)\sqcup\{x\}$. Moreover, $x$ is a birth--death critical point of $\widetilde f$.
\end{enumerate}
\end{lemma}

\begin{proof}
The local construction used below is the elementary birth construction of Laudenbach \cite[Definition~2.3 and Remark~2.4(1)]{Lau14}, originating in \cite[Chapter~III]{Cerf}.

Set $d:=\dim M$. Since $x$ is a regular point of $f$, after shrinking around $x$, the submersion theorem gives local coordinates
\[
(y,z)\in\mathbb R^{d-1}\times\mathbb R,
\qquad x=(0,0),\qquad\text{in which }f(y,z)=f(x)+z.
\]
We choose these coordinates on a neighborhood whose closure is contained in $U$.

The elementary birth construction provides a smooth function $G$ on the unit cylinder 
\(
C:=D^{d-1}\times[-1,1]
\) 
with the following properties:
\[
G(y,z)=z\quad\text{near }\partial C,\qquad\operatorname{Crit}(G)=\{0\},
\]
and, in a neighborhood of the origin,
\[
G(y,z)-G(0)=z^3+q(y),
\]
where $q$ is a nondegenerate quadratic form. 

For $r>0$ sufficiently small, the rescaled cylinder $rC$ is contained in the chosen coordinate neighborhood. Define $\widetilde f$ on $rC$ by
\[
\widetilde f(y,z):=f(x)+rG\left(\frac{y}{r},\frac{z}{r}\right),
\]
and set $\widetilde f=f$ outside $rC$. Since $G(Y,Z)=Z$ near $\partial C$, we have
\[
\widetilde f(y,z)=f(x)+r\frac{z}{r}=f(x)+z=f(y,z)
\]
near $\partial(rC)$. Hence the two definitions fit together smoothly, and $\widetilde f=f$ outside a compact subset of $U$.

For $(y,z)\in rC$, writing $(Y,Z)=(y/r,z/r)$, we obtain
\[
\left|\widetilde f(y,z)-f(y,z)\right|=r\left|G(Y,Z)-Z\right|.
\]
Consequently,
\[
\|\widetilde f-f\|_{C^0(M)}\leq r\|G-Z\|_{C^0(C)}.
\]
Choosing $r$ sufficiently small gives 
\(
\|\widetilde f-f\|_{C^0(M)}<\delta.
\)

Moreover,
\[
\dd\widetilde f_{(y,z)} =\dd G_{(y/r,z/r)}.
\]
Thus $\widetilde f$ has exactly one critical point in $rC$, namely $(y,z)=(0,0)$, which corresponds to $x$. Since $f$ has no critical points in the chosen coordinate neighborhood and $\widetilde f=f$ outside $rC$, it follows that
\[
\operatorname{Crit}(\widetilde f)=\operatorname{Crit}(f)\sqcup\{x\}.
\]

Finally, near the origin,
\[
\begin{aligned}\widetilde f(y,z)-\widetilde f(0,0)
&=r\left[G\left(\frac{y}{r},\frac{z}{r}\right)-G(0)\right] =\frac{z^3}{r^2}+\frac{1}{r}q(y).
\end{aligned}
\]
The quadratic form $r^{-1}q$ is still nondegenerate, and the cubic coefficient is nonzero. Therefore $x$ is a birth--death critical point of $\widetilde f$.
\end{proof}

Applying the local construction simultaneously in antipodal balls preserves oddness.
\begin{corollary}\label{cor:odd-insertion}
Let $F$ be a smooth odd function on $\Sph^{n-1}$ with finitely many critical points, and let its global minimum set be \(U=\operatorname{argmin}F.\) 

Then, for every integer $r\geq 1$, there is a smooth odd function $F^{(r)}$ whose global minimum set is still $U$ and which has exactly $r$ more antipodal pairs of critical points than $F$.
\end{corollary}

\begin{proof}
Choose $r$ regular points \(p_1,\ldots,p_r\) such that \(p_i\notin
\operatorname{Crit}(F)\). Since \(\operatorname{Crit}(F)\) is finite, we may choose sufficiently small neighborhoods \(B_i\) of \(p_i\) such that the \(2r\) balls \(B_i\) and \(-B_i\) are pairwise disjoint and \(\bigcup_{i=1}^r(B_i\cup -B_i)\cap\left(U\cup(-U)\cup\operatorname{Crit}(F)\right)=\varnothing\). In each \(B_i\), apply Lemma~\ref{lem:insertion} while in \(-B_i\) we define the modification by the odd rule:
\begin{align*}
F^{(r)}(u):=\begin{cases}\widetilde F_i(u),&u\in B_i,\\
-\widetilde F_i(-u),&u\in -B_i,\\
F(u),&u\notin\displaystyle\bigcup_i(B_i\cup-B_i).
\end{cases}
\end{align*}
Each operation inserts one new critical pair. The modifications can be made arbitrarily small in $C^0$ and supported in a compact set on which $F$ lies strictly above its minimum. Therefore the global minimum set \(U\) is preserved under the modification.
\end{proof}

We can now realize the entire finite spectrum.

\begin{theorem}
\label{thm:spectrum}
For every $n\geq 3$, 
\(
\Sigma_n^{\mathrm{fin}}=\{4,5,6,\ldots\}.
\) 
\end{theorem}

\begin{proof}
The inclusion $\Sigma_n^{\mathrm{fin}}\subseteq\{4,5,\ldots\}$ is the lower bound from \cite[Theorem~1.10]{PTW}. We only need to show the reverse inclusion.

Fix $m\geq 4$. Start with $F_4$ from Proposition~\ref{prop:smooth-four} and apply Corollary~\ref{cor:odd-insertion} with $r=m-4$. The resulting smooth odd function $F_m$ has exactly $m$ antipodal critical pairs, while its global minimum set remains the balanced triple $U_4$. Choose $0<\varepsilon<\varepsilon_0(F_m)$ and apply Proposition~\ref{thm:exact-local} to obtain
\[
           D_{K_{\eps,m}}
            =\frac{\omega_n}{2}+\eps F_m,\qquad \eps>0.
\]
Lemma~\ref{lem:balance} gives $p_{K_{\eps,m}}=0$, and Proposition~\ref{prop:known-bridge} identifies the $m$ critical pairs with exactly $m$ barycentric hyperplanes through that point. Since \(\varepsilon\) may be chosen arbitrarily small, the realizing body may be chosen arbitrarily \(C^\infty\)-close to the unit ball.
\end{proof}
\begin{remark}[The planar case]\label{rem:planar-spectrum}
Applying Corollary~\ref{cor:odd-insertion} with \(F_3\) in place of \(F_4\) also gives
\[
    \Sigma^{\mathrm{fin}}_2=\{3,4,5,\ldots\}.
\]
\end{remark}

\section{A uniform five-cut theorem}
\label{sec:five-cut}

The preceding sections concern barycentric hyperplanes through the Tukey median. We now show that this maximal-depth information has a further consequence. Namely, every convex body of dimension $n\geq 3$ possesses a free point on at least five barycentric hyperplanes. In dimension four this settles Conjecture~\ref{conj:grunbaum-general}. 

\subsection{The section-centroid map and its global degree}

Let
\[
 \mathcal A_K:=\{H:\ H\text{ is an unoriented affine hyperplane, }
                         H\cap\operatorname{int}K\ne\varnothing\}\subset_{open} \frac{\Sph^{n-1}\times\R}{(u,t)\sim(-u,-t)}.
\]
In particular, $\mathcal A_K$ is an non-compact $n$-manifold. Define the section-centroid map on $\mathcal A_K$ by 
\[
 \Gamma_K:\mathcal A_K\longrightarrow\operatorname{int}K,
 \qquad
 \Gamma_K(H):=c(K\cap H).
\]

Section-center maps of this form have appeared previously in the multiplicity approach to Gr\"unbaum-type problems \cite{Kar11,BK16}. Here we work directly on the non-compact manifold \(\mathcal A_K\).

\begin{lemma}\label{lem:centroid-map}
The map $\Gamma_K$ is continuous. For every $p\in\operatorname{int}K$, \(\Gamma_K^{-1}(p)=\mathcal B_K(p).\) In particular, every fiber is compact.
\end{lemma}

\begin{proof}
If $H_j\to H$ in $\mathcal A_K$, choose orthogonal identifications of $H_j$ with $H$ converging to the identity. Then \(\mathbf 1_{K\cap H_j} \to \mathbf 1_{K\cap H}\) almost everywhere while all sections remain in one compact set. Dominated convergence theorem gives $c(K\cap H_j)\to c(K\cap H)$. This proves continuity.

The second assertion is the definition of $\mathcal B_K$. Finally, each fiber $\Gamma^{-1}(p)$ is a closed subset of \(\mathbb{RP}^{n-1}\), hence compact.
\end{proof}

For an interior point $p$, we may introduce the following lemma:
\begin{lemma}\label{lem:global-centroid-degree}
For every convex body $K\subset\R^n$ and every $p\in\operatorname{int}K$,
\[
 \deg_2(\Gamma_K,\mathcal A_K,p)=1.
\]
Here the degree is relative Brouwer degree: the full fiber is enclosed in a relatively compact open subset of $\mathcal A_K$.
\end{lemma}

\begin{proof}
Choose $o\in\operatorname{int}K$ with $o\ne p$, and let $Q_o(H)$ be the orthogonal projection of $o$ onto $H$.  Consider
\[
 G_s(H):=(1-s)\Gamma_K(H)+sQ_o(H),\qquad 0\le s\le1.
\]
The entire path \(s\mapsto G_s(H)\) is contained in $H$. Thus \(G_s(H)=p\Rightarrow H\in\mathcal P_p=:\{H\in\mathcal A_K:p\in H\}\simeq\mathbb{RP}^{n-1}\subset\mathcal A_K\). The solutions of the \(G_s(H)=p\) are therefore enclosed in one relatively compact open subset, so relative degree is homotopy invariant.

The equation $Q_o(H)=p$ has exactly one unoriented solution: the hyperplane through $p$ orthogonal to $p-o$. In oriented coordinates,
\[
 Q_o(H(u,t))=o+(t-\ip{o}{u})u.
\]
At the solution $u=(p-o)/|p-o|$ and $t=\ip{p}{u}$, its differential is
\[
 (\dot u,\dot t)\longmapsto
 |p-o|\dot u+(\dot t-\ip{o}{\dot u})u,
\]
which is an isomorphism from $u^\perp\oplus\R$ to $\R^n$. The unique solution is regular, and therefore the mod-$2$ degree of both $Q_o$ and $\Gamma_K$ is 1.
\end{proof}

\subsection{Local degree at a depth-minimizing cut}

\begin{lemma}\label{lem:minimizing-degree}
Let $H_0=p+u_0^\perp$ be isolated in $\Gamma_K^{-1}(p)$. If the oriented normal $u_0$ is a strict local minimum of the halfspace-volume functional, then
\[
 \deg_2(\Gamma_K,H_0;p)=1.
\]
\end{lemma}

\begin{proof}
Translate $p$ to the origin and rotate $u_0$ to $e_n$.  Re-parameterize nearby hyperplanes by
\[
 H(y,s)=\{(x',x_n):x_n=\ip{y}{x'}+s\}.
\]
Write
\[
 \Gamma_K(H(y,s))
   =\bigl(r(y,s),\,\ip{y}{r(y,s)}+s\bigr).
\]
The zero-preserving homotopies
\[
 (r(y,s),s+\tau\ip{y}{r(y,s)}),
 \qquad
 (r(y,\tau s),s),\qquad 0\le\tau\le1,
\]give the following reduction:
\[\deg_2(\Gamma_K,H_0;0)=
\deg_2\big((y,s)\mapsto(r(y,0),s),(0,0);0\big).\]
Isolation ensures that no zero crosses the boundary of a sufficiently small coordinate neighborhood.

Set
\[
 f(y):=\operatorname{vol}\bigl(K\cap\{x_n\ge\ip{y}{x'}\}\bigr),
 \qquad
 A(y):=\operatorname{vol}_{n-1}(K\cap H(y,0)).
\]
\cite[(5)]{PTW} gives
\[
 \nabla f(y)=-\frac{A(y)}{\sqrt{1+|y|^2}}\,r(y,0),\qquad A(y)>0.
\]
The origin is an isolated strict local minimum of $f$ and \(\deg_2(\nabla f,0;0)=1\) \cite{deg1}. Multiplication by nowhere-zero scalar does not change mod-$2$ degree, which proves the assertion.
\end{proof}

The following elementary consequence of invariance of domain supplies the additional branch.

\begin{lemma}\label{lem:degree-zero-split}
Let $F:(M^n,x_0)\to(N^n,y_0)$ be a continuous map germ.  Suppose that $x_0$ is isolated in $F^{-1}(y_0)$ and \(\deg_2(F,x_0;y_0)=0.\) Then every neighborhood of $x_0$ contains two distinct points with the same image. Their common image may be chosen in any neighborhood of $y_0$.
\end{lemma}

\begin{proof}
Suppose by contradiction that $F$ is an injective on a neighborhood of $x_0$, then its restriction is a homeomorphism onto an open set. 
Excision would then give local mod-$2$ degree one, a contradiction. Hence $F$ is non-injective in every neighborhood; For second assertion, given a neighborhood $W$ of $y_0$, choose a source neighborhood $V$ with\[
F(V)\subset W,\qquad
F^{-1}(y_0)\cap V=\{x_0\}.\] 
Noninjectivity gives distinct $x',x''\in V$ with $F(x')=F(x'')\in W$.
\end{proof}

\subsection{The five-cut theorem}

\begin{theorem}\label{thm:five-cut}
Every convex body $K\subset\R^n$, $n\ge3$, admits a point $p\in\operatorname{int}K$ through which pass at least five distinct unoriented barycentric hyperplanes.
\end{theorem}

\begin{proof}
Let $p_0$ be the Tukey median. By \cite{PTW}, at least four barycentric hyperplanes pass through $p_0$, and at least three distinct ones arise from depth-minimizing directions. If $\#\mathcal B_K(p_0)\ge5$, we are done. We therefore assume
\[
 \Gamma_K^{-1}(p_0)=\{H_1,H_2,H_3,H_4\},
\]
where $H_1,H_2,H_3$ are depth-realizing.

Each selected minimizing direction is an isolated strict minimum. Lemma~\ref{lem:minimizing-degree} therefore yields
\[
 \deg_2(\Gamma_K,H_i;p_0)=1,\qquad i=1,2,3.
\]
Additivity of local degree and Lemma~\ref{lem:global-centroid-degree} give
\[
 1=\deg_2(\Gamma_K,\mathcal A_K,p_0)
  =1+1+1+\deg_2(\Gamma_K,H_4;p_0)\pmod2.
\]
Thus the fourth local degree is zero.

Choose pairwise disjoint relatively compact isolating coordinate neighborhoods $U_i$ of $H_i$.  For $i=1,2,3$, local degree one persists for every target in a sufficiently small common neighborhood $W\Subset\operatorname{int}K$ of $p_0$. In particular, every $p\in W$ has a preimage in each $U_i$. Apply
Lemma~\ref{lem:degree-zero-split} to the germ of $\Gamma_K$ at $H_4$. It supplies distinct $H_4',H_4''\in U_4$ and a point $p\in W$ such that
\[
 \Gamma_K(H_4')=\Gamma_K(H_4'')=p.
\]
The three persistent degree-one branches supply $H_i'\in U_i$ with $\Gamma_K(H_i')=p$, $i=1,2,3$. Since the $U_i$ are disjoint, these are five distinct unoriented hyperplanes. Finally, $\Gamma_K(H)=p$ implies $p\in H$, so all five pass through $p$ and are barycentric there.
\end{proof}

\begin{corollary}
\label{cor:grunbaum-four}
Conjecture \ref{conj:grunbaum-general} is true in dimensions $2\le n\le 4$. 
\end{corollary}

\begin{remark}
The proof does not iterate automatically. The new five-fold point need not maximize halfspace depth, so the three protected minimizing branches used above are no longer available.  Thus the theorem does not settle the required $n+1$ bound in dimensions $n\ge5$.
\end{remark}

\begin{remark}
  Earlier multiplicity results for section-center maps give dimension-dependent lower bounds and may yield more than five coincident sections in certain special dimensions \cite{Kar11}. The point here is the uniform lower bound for all dimensions $n\ge 3$.
\end{remark}

\section{A centered refinement of the exact realization theorem}\label{sec:centered}

Proposition~\ref{thm:exact-local} prescribes the halfspace volume functional at the origin but does not control the centroid of the realizing body. We now show that the degree-one component is the only first-order obstruction to imposing both requirements exactly.

The relevance of the degree-one component already appears in \cite[Section 3]{XiongYang}. The new point here is to introduce an even mean-zero correction in $\ker T$. This leaves $D_K$ unchanged while allowing the centroid constraint to be imposed exactly.

\begin{theorem}
\label{thm:centered}
Let $n\geq 2$ and let $0\neq F\in C^\infty_{\mathrm{odd}}(\Sph^{n-1})$. Assume that the degree-one spherical harmonic component of $F$ vanishes. Then there is $\eps_0>0$ such that, for every $0<|\eps|<\eps_0$, there is a smooth, positively curved convex body $K_\eps$ with
\[
 \cK(K_\eps)=0,\qquad
 \vol(K_\eps)=\omega_n,\qquad
 K_\eps\longrightarrow B^n\quad\text{in }C^\infty,
\]
and
\begin{equation}\label{eq:centered-exact}
       D_{K_\eps}(\xi)
        =\frac{\omega_n}{2}+\eps F(\xi)
        \quad\text{for every }\xi\in\Sph^{n-1}.
\end{equation}
Conversely, let \(\epsilon\mapsto\rho_\epsilon\) be differentiable at \(\epsilon=0\) as a map into \(C^2(S^{n-1})\), with \(\rho_0\equiv1\), and suppose that the associated radial bodies have centroid zero and satisfy \eqref{eq:centered-exact} for every sufficiently small \(\epsilon\). Then the degree-one component of $F$ must vanish. 
\end{theorem}

\begin{proof}
\textbf{Sufficiency. }Let $g=T^{-1}F$. By Theorem~\ref{thm:rubin}, $g$ is smooth, odd, nonzero, and has no degree-one component. Hence
\begin{equation}\label{eq:first-moment-g}
             \int_{\Sph^{n-1}}u g(u)\dd\sigma(u)=0.
\end{equation}
For $a\in\R^n$, define the even function
\[
 h_a(u):=\ip{a}{u}g(u).
\]
~\eqref{eq:first-moment-g} gives
$\int h_a\dd\sigma=0$, so \eqref{eq:T-even} yields $Th_a=0$.

Set
\begin{equation}\label{eq:centered-q}
 q_{\eps,a}(u)
   :=1+n\eps g(u)+\eps^2h_a(u),
 \qquad
 \rho_{\eps,a}:=q_{\eps,a}^{1/n}.
\end{equation}
Whenever $q_{\eps,a}>0$, its radial body has, exactly,
\[
 \vol(K_{\eps,a})=\frac1n\int_{\mathbb S^{n-1}} q_{\eps,a}\dd \sigma=\frac1n\left(n\omega_n+n\eps\int_{\mathbb S^{n-1}}g\dd\sigma+\eps^2\int_{\mathbb S^{n-1}}h_a\dd\sigma\right)=\omega_n\]\[
 D_{K_{\eps,a}}
  =\frac1nTq_{\eps,a}=
\frac1n(T1+n\eps Tg+\eps^2 Th_a)
  =\frac{\omega_n}{2}+\eps F.
\]
It remains to choose $a=a(\eps)$ so that the centroid vanishes.

Put $\alpha=(n+1)/n$ and define:
\[
 C(\eps,a):=\int_{\Sph^{n-1}}u\,q_{\eps,a}(u)^\alpha\dd\sigma(u).
\]
Thus \[\int_{K_{\epsilon,a}}x\,dx
=
\frac1{n+1}C(\epsilon,a),\qquad c(K_{\epsilon,a})=0\iff C(\epsilon,a)=0.\]

The coefficients of $\eps^0$, $\eps^1$, and $\eps^2$ in $C$ vanish: this follows respectively from antipodal symmetry, \eqref{eq:first-moment-g}, and the parity of $h_a$ and $g^2$. Indeed, denote $C_0,\,C_1,\,C_2$ to be the corresponding coefficients, then:\begin{align*}
  C_0&=\int_{\mathbb S^{n-1}}u\,d\sigma=0,\\
  C_1&=\alpha n\int ug(u)d\sigma=0,\\
  C_2&=\alpha\int u h_a(u)d\sigma+\frac{\alpha(\alpha-1)n^2}{2}\int ug(u)^2d\sigma=0.
\end{align*}

Therefore \(C(\eps,a)=\eps^3\Phi(\eps,a)\) for a smooth map $\Phi$ near $\eps=0$. Taylor expansion gives
\begin{equation}\label{eq:Phi-zero}
 \Phi(0,a)
   =\alpha M_g a-\frac{n^2-1}{6}\,b_g,
\end{equation}
where
\[
 M_g:=\int_{\Sph^{n-1}}u\otimes u\,g(u)^2\dd\sigma(u),
 \qquad
 b_g:=\int_{\Sph^{n-1}}u\,g(u)^3\dd\sigma(u).
\]
Indeed, two order-three terms come from \(\frac{\alpha (\alpha-1)}{2} 2n\varepsilon^3gh_a\) and \(\frac{\alpha (\alpha-1)(\alpha-2)}{6}n^3\varepsilon^3g^3\), direct calculation shows that: 
\begin{align*}
  \alpha(\alpha-1)n\int ug h_a\dd\sigma&=\alpha(\alpha-1)n\int u\langle a,u\rangle g(u)^2d\sigma=\alpha(\int u\otimes ug(u)^2d\sigma)a=\alpha M_ga\\
   \frac{\alpha(\alpha-1)(\alpha-2)n^3}{6}\int ug^3\dd\sigma&=\frac16\cdot\frac{n+1}{n}\cdot\frac1n\cdot\left(-\frac{n-1}{n}\right)\cdot n^3\int ug^3\dd\sigma=-\frac{n^2-1}{6}b_g
\end{align*}

The matrix $M_g$ is positive definite.  For $v\neq0$,
\[
 v^{\top}M_gv=\int v^\top(u\otimes u)vg(u)^2d\sigma=\int_{\Sph^{n-1}}\ip{v}{u}^2g(u)^2\dd\sigma(u)>0,
\]
because a nonzero smooth function is nonzero on an open set, which cannot be contained in the great sphere $v^\perp$. Consequently \eqref{eq:Phi-zero} has the unique solution
\[
             a_0=\frac{n(n-1)}{6}M_g^{-1}b_g,
\]
and $\partial_a\Phi(0,a_0)=\alpha M_g$ is invertible. The implicit function theorem gives a smooth bounded $a(\eps)$, with $a(0)=a_0$, such that $\Phi(\eps,a(\eps))=0$.  Thus $C(\eps,a(\eps))=0$.

Since $a(\eps)$ remains bounded, $q_{\eps,a(\eps)}\to1$ and $\rho_{\eps,a(\eps)}\to1$ in $C^\infty$. For small $|\eps|$, positivity and the same open curvature argument used in Proposition~\ref{thm:exact-local} complete the construction.

\medskip
\noindent
\textbf{Necessity. }Write
\[
 q_\eps:=\rho_{K_\eps}^n
       =1+\eps \frac{d}{d\varepsilon}
\rho_{K_\varepsilon}^n
\bigg|_{\varepsilon=0}+o(\eps)=:1+\eps q_1+o(\eps),
\]
and let \(q_1^\mathrm{odd}(u)=\frac{q_1(u)-q_1(-u)}2\) be its odd part. Differentiating \eqref{eq:profile-transform} and taking odd parts gives
\[
              Tq_1^{\mathrm{odd}}=nF.
\]
Hence Theorem~\ref{thm:rubin} implies $q_1^{\mathrm{odd}}=ng$, where $g=T^{-1}F$. $\cK(K_\eps)=0$ gives:
\begin{align*}
0&=\int_{\mathbb S^{n-1}} uq_\varepsilon(u)^\alpha\dd\sigma(u),\qquad \alpha=\frac{n+1}{n},\\
&=\int_{\mathbb S^{n-1}}u\left(1+\alpha\varepsilon q_1(u)+o(\varepsilon)\right)\dd\sigma(u),\\
&=\int_{\mathbb S^{n-1}}u\dd\sigma(u)+\alpha\varepsilon\int_{S^{n-1}}u q_1(u)\dd\sigma(u)+o(\varepsilon),\\
&=\alpha\varepsilon\int_{\mathbb S^{n-1}}u q_1(u)\dd\sigma(u)+o(\varepsilon).
\end{align*}
Differentiating at $\eps=0$ gives:
\[
 0
 =\alpha\int_{\Sph^{n-1}}u q_1(u)\dd\sigma(u)
 =\alpha\int_{\Sph^{n-1}}u (q_1^{\mathrm{odd}}(u)+q_1^{\mathrm{even}}(u))\dd\sigma(u)
 =\alpha n\int_{\Sph^{n-1}}u g(u)\dd\sigma(u),
\]
$q_1^{\mathrm{even}}$ contributes zero in the integral by oddness of $uq_1^{\mathrm{even}}$. Thus $g$ has no degree-one component. Since $T$ acts by a nonzero scalar on the degree-one harmonic subspace, $F=Tg$ also has no degree-one component.
\end{proof}

\begin{remark}
For $F=0$, the conclusion is realized by the unit ball.
\end{remark}

\begin{flushleft}
\medskip\noindent
\begin{tabbing}

				Xiaoxiang Jiao\\
				School of Mathematical Sciences, University of Chinese Academy of Sciences\\
				19A Yuquan Road, Beijing, 100049, China\\
				\texttt{xxjiao@ucas.ac.cn}
				
\end{tabbing}

\begin{tabbing}
Hangyue Zhu\\School of Mathematical Sciences, University of Chinese Academy of Sciences\\
				19A Yuquan Road, Beijing, 100049, China\\
				\texttt{zhuhangyue24@mails.ucas.ac.cn}
\end{tabbing}

\end{flushleft}

\paragraph{Conflict of Interest}
On behalf of all authors, the corresponding author states that there is no conflict of interest.

\paragraph{Data Availability}
Data sharing is not applicable to this article, as no datasets were generated or analyzed during the current study.

\paragraph{Funding}
The authors are supported by the National Natural Science Foundation of China (Grant no. 12371055). 
\bibliographystyle{alpha}
\bibliography{ref}
\end{document}